\documentclass[12pt,reqno]{amsart}

\usepackage{pdfsync}

\usepackage{textcomp}

\usepackage{amssymb}

\usepackage{mathrsfs}

\DeclareMathAlphabet{\mathpzc}{OT1}{pzc}{m}{it}

\usepackage[all]{xy}

\entrymodifiers={+!!<0pt,\fontdimen22\textfont2>}

\def\BZ{\mathbb{Z}}

\def\sA{\mathsf{A}}
\def\sB{\mathsf{B}}
\def\sC{\mathsf{C}}
\def\sD{\mathsf{D}}

\def\sH{\mathsf{H}}

\def\sM{\mathsf{M}}
\def\sN{\mathsf{N}}

\def\sS{\mathsf{S}}
\def\sT{\mathsf{T}}

\def\sX{\mathsf{X}}
\def\sY{\mathsf{Y}}

\def\add{\operatorname{add}}

\def\adots{\mathinner{\mkern1mu\raise1.0pt\vbox{\kern7.0pt\hbox{.}}\mkern2mu\raise4.0pt\hbox{.}\mkern2mu\raise7.0pt\hbox{.}\mkern1mu}}

\def\D{\sD}
\def\Dc{\D^{\operatorname{c}}}

\def\Df{\D^{\operatorname{f}}}
\def\dim{\operatorname{dim}}

\def\gr{\mathsf{gr}}

\def\H{\operatorname{H}}

\def\Hom{\operatorname{Hom}}
\def\id{\operatorname{id}}

\def\Ker{\operatorname{Ker}}

\newcommand\LTensor[1]{\overset{{\rm L}}{\underset{#1}{\otimes}}}

\def\opp{\operatorname{op}}

\def\RHom{\operatorname{RHom}}

\newtheorem{Lemma}{Lemma}[section]
\newtheorem{Proposition}[Lemma]{Proposition}

\theoremstyle{definition}
\newtheorem{Definition}[Lemma]{Definition}

\newtheorem{Remark}[Lemma]{Remark}

\begin{document}

\setlength{\parindent}{0pt}
\setlength{\parskip}{7pt}
%The default \baselineskip is close to 4.8mm
%\setlength{\baselineskip}{5.3mm}

\title[t-structures and CY dimension]{Sparseness of t-structures and
negative Calabi--Yau dimension in triangulated categories generated by
a spherical object}

\author{Thorsten Holm}
\address{Institut f\"{u}r Algebra, Zahlentheorie und Diskrete
  Mathematik, Fa\-kul\-t\"{a}t f\"{u}r Ma\-the\-ma\-tik und Physik, Leibniz
  Universit\"{a}t Hannover, Welfengarten 1, 30167 Hannover, Germany}
\email{holm@math.uni-hannover.de}
\urladdr{http://www.iazd.uni-hannover.de/\~{ }tholm}

\author{Peter J\o rgensen}
\address{School of Mathematics and Statistics,
Newcastle University, Newcastle upon Tyne NE1 7RU, United Kingdom}
\email{peter.jorgensen@ncl.ac.uk}
\urladdr{http://www.staff.ncl.ac.uk/peter.jorgensen}

\author{Dong Yang}
\email{dongyang2002@googlemail.com}

%\author{Next author goes here}
%\address{Next author's postal address goes here}
%\email{Next author's mail address goes here}
%\urladdr{}

%\thanks{Date: \today. A thank you would go here}


%13D25: Complexes
%16E10: Homological dimension
%16E35: Derived categories
%16E45: Differential graded algebras and applications
%16G10: Representations of Artinian rings 
%16G60: Representation type (finite, tame, wild, etc.) 
%16G70: Auslander-Reiten sequences (almost split sequences) and
%       Auslander-Reiten quivers
%16S90: Torsion theories; radicals on module categories
%18E30: Derived categories, triangulated categories
%18E35: Localization of categories
%18E40: Torsion theories, radicals
%18G05: Projectives and injectives
%18G35: Chain complexes
%18G99: Homological algebra: None of the above, but in this section 
%55P62: Rational homotopy theory

\begin{abstract} 

Let $k$ be an algebraically closed field and let $\sT$ be the
$k$-linear algebraic triangulated category ge\-ne\-ra\-ted by a
$w$-spherical object for an integer $w$.  For certain values of $w$
this category is classical.  For instance, if $w = 0$ then it is the
compact derived category of the dual numbers over $k$.

As main results of the paper we show that for $w \leq 0$, the category
$\sT$ has no non-trivial t-structures, but does have one family of
non-trivial co-t-structures, whereas for $w \geq 1$ the opposite
statement holds.

Moreover, without any claim to originality, we observe that for $w
\leq -1$, the category $\sT$ is a candidate to have negative
Calabi--Yau dimension since $\Sigma^w$ is the unique power of the
suspension functor which is a Serre functor.

\end{abstract}

\maketitle

\setcounter{section}{-1}
\section{Introduction}
\label{sec:introduction}

Let $k$ be an algebraically closed field, $w$ an integer, and let
$\sT$ be a $k$-linear algebraic triangulated category which is
idempotent complete and classically generated by a $w$-spherical
object.

The categories $\sT$, examined initially in \cite{J} for $w \geq 2$,
have recently been of considerable interest, see \cite{FY},
\cite{HJ}, \cite{KYZ}, \cite{Ng}, and \cite{ZZ}.  The purpose of this
paper is twofold.

First, we show the following main result.

{\bf Theorem A. }
{\em
If $w \leq 0$, then $\sT$ has no non-trivial t-structures.  It has
one family of non-trivial co-t-structures, all of which are
(de)suspensions of a canonical one.

If $w \geq 1$, then $\sT$ has no non-trivial co-t-structures.  It has
one family of non-trivial t-structures, all of which are
(de)suspensions of a canonical one.
}

For $w \leq 0$ this is a particularly clean instance of Bondarko's
remark \cite[rmk.\ 4.3.4.4]{B} that there are sometimes ``more''
co-t-structures than t-structures in a triangulated category.
Note that the case $w = 2$ is originally due to Ng \cite[thms.\ 4.1
and 4.2]{Ng}.

Secondly, without any claim to originality, we observe that if $w \leq
-1$ then $\sT$ is a candidate for having negative Calabi--Yau
dimension, although there does not yet appear to be a universally
accepted definition of this concept.  Namely, the $w$'th power of the
suspension functor, $\Sigma^w$, is a Serre functor for $\sT$, and
$\Sigma^w$ is the {\em only} power of the suspension which is a Serre
functor.  For $w \geq 2$ this is contained in \cite[prop.\ 6.5]{J}.
For a general $w$ it is well known to the experts; we show an easy
proof in Proposition \ref{pro:Serre}.

The proof of Theorem A occupies Section \ref{sec:proof} while Sections
\ref{sec:T} to \ref{sec:remarks} are preparatory.  Let us end the
introduction by giving some background and explaining the terms used
above.

\subsection{What is $\sT$?}

$\,$
\medskip
\newline
For certain small values of $w$, the category $\sT$ is well known in
different guises: For $w = 0$ it is $\Dc(k[X] / (X^{2}))$, the compact
derived category of the dual numbers.  For $w = 1$ it is
$\Df(k\mbox{\textlbrackdbl} X \mbox{\textrbrackdbl})$, the derived
category of complexes with bounded finite length homology over the
formal power series ring.  And for $w = 2$ it is the cluster category
of type $A_{\infty}$, see \cite{HJ}.  For $w$ negative, $\sT$ is less
classical.

In general, $\sT$ is determined up to triangulated equivalence by the
properties stated in the first paragraph of the paper by \cite[thm.\
2.1]{KYZ}.  We briefly explain these properties:

A triangulated category is algebraic if it is the stable category of a
Frobenius category; see \cite[sec.\ 9]{H}.

An additive category $\sA$ is idempotent complete if, for each
idempotent $e$ in an endomorphism ring $\sA( a , a )$, we have $e =
\iota\pi$ where $\iota$ and $\pi$ are the inclusion and projection of
a direct summand of $a$.  Note that $\sA( - , - )$ is shorthand for
$\Hom_{\sA}( - , - )$.

A $w$-spherical object $s$ in a $k$-linear triangulated category $\sS$
is defined by having graded endomorphism algebra 
$\sS( s , \Sigma^* s )$ isomorphic to $k[X]/(X^2)$ with $X$ placed in
cohomological degree $w$.

A triangulated category $\sS$ is classically generated by an object
$s$ if each object in $\sS$ can be built from $s$ using finitely many
(de)suspensions, distinguished triangles, and direct summands.

\subsection{t-structures and co-t-structures}

$\,$
\medskip
\newline
To explain these, we first introduce the more fundamental notion of a
torsion pair in a triangulated category due to Iyama and Yoshino
\cite[def.\ 2.2]{IY}. 

If $\sS$ is a triangulated category, then a torsion pair in $\sS$ is a
pair $( \sM , \sN )$ of full subcategories closed under direct sums
and summands, satisfying that $\sS( \sM , \sN ) = 0$ and that $\sS =
\sM * \sN$ where $\sM * \sN$ stands for the class of objects $s$
appearing in distinguished triangles $m \rightarrow s \rightarrow n$
with $m \in \sM$, $n \in \sN$.

A torsion pair $( \sM , \sN )$ is called a t-structure if $\Sigma \sM
\subseteq \sM$, and a co-t-structure if $\Sigma^{-1}\sM \subseteq
\sM$.  In each case, the structure is called trivial if it is
$( \sS , 0 )$ or $( 0 , \sS )$ and non-trivial otherwise.

This is not how t-structures and co-t-structures were first defined by
Beilinson, Bernstein, and Deligne in \cite[def.\ 1.3.1]{BBD},
respectively by Bondarko and Pauksztello in \cite[def.\ 1.1.1]{B} and
\cite[def.\ 2.4]{P}, but it is an economical way to present them and
to highlight their dual natures.

t-structures have become classical objects of homological algebra
while co-t-structures were introduced more recently.  They both enable
one to ``slice'' objects of a triangulated category into simpler bits
and they are the subject of vigorous research.

\subsection{Silting subcategories}

$\,$
\medskip
\newline
We are grateful to Chang\-ji\-an Fu for the following observation:
$\Sigma^w$ is a Serre functor of $\sT$.  In the terminology of
\cite{AI} this means that $\sT$ is $w$-Calabi--Yau.  Moreover, $\sT$
is generated by a $w$-spherical object $s$; in particular, for $w \leq
-1$ we have $\sT( s , \Sigma^{>0}s ) = 0$.  In the terminology of
\cite{AI}, this means that $s$ is a silting object.

So for $w \leq -1$, the category $\sT$ is $w$-Calabi--Yau with the
silting subcategory $\add(s)$.  The existence of a category with these
properties was left as a question at the end of \cite[sec.\ 2.1]{AI}.

It is not hard to check directly that for $w \leq -1$, the basic
silting objects in $\sT$ are precisely the (de)suspensions of $s$.
This also follows from \cite[thm.\ 2.26]{AI}.

\section{Basic properties of $\sT$}
\label{sec:T}

None of the material of this section is original, but not all of it is
given explicitly in the original references \cite{FY}, \cite{J}, and
\cite{KYZ}.  We give a brief, explicit presentation to facilitate
the rest of the paper.

\begin{Remark}
\label{rmk:Krull-Schmidt}
The category $\sT$ is Krull-Schmidt by \cite[p.\ 52]{R}.  Namely, it
is idempotent complete by assumption, and it has finite dimensional
$\Hom$ spaces because each object is finitely built from a
$w$-spherical object $s$ which in particular satisfies $\dim_k \sT( s
, \Sigma^i s ) < \infty$ for each $i$.
\end{Remark}

We need to compute inside $\sT$.  Hence a concrete model is more
useful than an abstract characterisation.  Let us redefine $\sT$ as
such a model.

\begin{Definition}
Set $d = w - 1$ and consider the polynomial ring $k[T]$ as a
Differential Graded (DG) algebra with $T$ in homological degree $d$
and zero differential.  We denote this DG algebra by $A$.

Consider $\sD(A)$, the derived category of DG left-$A$-modules, and
let $\sT$ be $\langle k \rangle$, the thick sub\-ca\-te\-go\-ry
generated by the trivial DG module $k = A / (T)$ where $(T)$ is the DG
ideal generated by $T$.
\end{Definition}

This is how $\sT$ will be defined for the rest of the paper, except in
the proof of Proposition \ref{pro:Serre}.  It is compatible with the
previous definition of $\sT$ by the following result.

\begin{Lemma}
The category $\sT = \langle k \rangle$ is a $k$-linear algebraic
triangulated category which is idempotent complete and classically
generated by the $w$-spherical object $k$.
\end{Lemma}

\begin{proof}
The only part which is not clear is that $k$ is $w$-spherical.  But
there is a distinguished triangle
\begin{equation}
\label{equ:k}
  \Sigma^d A \stackrel{\cdot T}{\longrightarrow} A \longrightarrow k
\end{equation}
in $\sD(A)$, induced by the corresponding
short exact sequence of DG modules.  Applying $\RHom_A(-,k)$ gives
another distinguished triangle whose long exact homology sequence
shows that $k$ is a $w$-spherical object of $\sD(A)$.
\end{proof}

\begin{Remark}
The distinguished triangle \eqref{equ:k} also shows that $k$ is a
compact object of $\sD(A)$, so $\sT$ is even a subcategory of the
compact derived category $\Dc(A)$.
\end{Remark}

\begin{Definition}
For each $r \geq 0$, the element $T^{r+1}$ of $A$ generates a DG ideal
$(T^{r + 1})$.  Consider the quotient $X_{r} = A / (T^{r + 1})$ as a
DG left-$A$-module.
\end{Definition}

\begin{Remark}
\label{rmk:Xr}
There is a distinguished triangle
\[
  \Sigma^{(r+1)d} A
  \stackrel{\cdot T^{r+1}}{\longrightarrow} A
  \longrightarrow X_r
\]
in $\Dc(A)$, induced by the corresponding short exact sequence of DG
modules. 
\end{Remark}

\begin{Proposition}
\label{pro:Xr}
The indecomposable objects of $\sT$ are precisely the
(de)sus\-pen\-si\-ons of the objects $X_r$.
\end{Proposition}

\begin{proof}
Note that $\Dc(A) = \langle A \rangle$ and that
$\Hom_{\Dc(A)}(A,\Sigma^* A)$ is isomorphic to $k[T]$ as a graded
algebra, where $T$ is still in homological degree $d$.  Since
$\gr(k[T]^{\opp})$, the abelian category of finitely generated graded
right-$k[T]$-modules and graded homomorphisms, is hereditary,
\cite[thm.\ 3.6]{KYZ} says that the functor
\[
  \Hom_{\Dc(A)}(A,\Sigma^*(-)) = \H^*(-)
\]
induces a bijection between the isomorphism classes of indecomposable
objects of $\Dc(A)$ and $\gr(k[T]^{\opp})$.  This has the following
consequences.

If $w \neq 1$ then $d \neq 0$.  Then up to isomorphism, the
indecomposable objects of $\gr(k[T]^{\opp})$ are precisely the graded
shifts of the graded modules $k[T]$ and $k[T] / (T^{r+1})$ for $r \geq
0$.  So up to isomorphism, the indecomposable objects of $\Dc(A)$ are
the (de)suspensions of $A$ and the objects $X_r$ for $r \geq 0$.

Of these objects, precisely the $X_{r}$ are in $\sT$, so up to
isomorphism the indecomposable objects of $\sT$ are the
(de)suspensions of the objects $X_r$ for $r \geq 0$.

If $w = 1$ then $d = 0$ so $A$ and $k[T]$ are concentrated in degree
$0$.  A graded right-$k[T]$-module is the direct sum of its graded
components, and it follows that the indecomposable objects of
$\gr(k[T]^{\opp})$ are the indecomposable ungraded
right-$k[T]$-modules placed in a single graded degree.  But up to
isomorphism, these are $k[T]$ and $k[T] / (f(T))$ where $f(T)$ is a
power of an irreducible, hence first degree, polynomial.  So up to
isomorphism, the indecomposable objects of $\Dc(A)$ are the
(de)suspensions of the objects $A$ and $A / (f(T))$ viewed in
$\Dc(A)$.

Again, of these objects, precisely the $X_{r}$ are in $\sT$, so up to
isomorphism the indecomposable objects of $\sT$ are the
(de)suspensions of the objects $X_r$ for $r \geq 0$.
\end{proof}

It is not hard to see that $A$ is the $w$-Calabi--Yau completion of
$k$ in the sense of \cite[4.1]{K}.  As a consequence, $\sT = \langle k
\rangle$ has Serre functor $S = \Sigma^w$. Here we give a direct proof
of this fact.

\begin{Proposition}
\label{pro:Serre}
The category $\sT$ has Serre functor $S = \Sigma^w$, and this is the
only power of the suspension which is a Serre functor.
\end{Proposition}

\begin{proof}
For this proof only, it is convenient to use another model for $\sT$.
Consider the dual numbers $k[U] / (U^2)$ and view them as a DG algebra
with $U$ placed in cohomological degree $w$ and zero differential.
Denoting this DG algebra by $B$, it is immediate that $B$ is a
$w$-spherical object of $\sD(B)$, the derived category of DG
left-$B$-modules, and so the thick subcategory $\langle B \rangle$
generated by $B$ is equivalent to $\sT$.  This is the model we will
use.  In fact, $\langle B \rangle$ is equal to the compact derived
category $\Dc(B)$.

For $X, Y \in \Dc(B)$ we have the following natural isomorphisms
where $D(-) = \Hom_k(-,k)$. 
\begin{align*}
  D \RHom_B( Y , DB \LTensor{B} X )
  & \stackrel{\rm (a)}{\cong} D \big( \RHom_B( Y , DB ) \LTensor{B} X \big) \\
  & \stackrel{\rm (b)}{\cong} \RHom_{B^{\opp}} \big( \RHom_B( Y , DB ) , DX \big) \\
  & \stackrel{\rm (c)}{\cong} \RHom_{B^{\opp}}( DY , DX ) \\
  & \stackrel{\rm (d)}{\cong} \RHom_B( X , Y ).
\end{align*}
Here (a) holds for $X = B$ and hence for the given $X$ because it is
finitely built from $B$.  The isomorphisms (b) and (c) are by
adjointness of $\LTensor{}$ and $\RHom$.  And (d) is duality.

Taking zeroth homology of the above formula shows that $DB \LTensor{B}
-$ is a right Serre functor of $\Dc(B)$.  But direct computation shows
$DB \cong \Sigma^w B$ as DG $B$-bimodules, so $\Sigma^w$ is a right
Serre functor.  Since it is an equivalence of categories, it is even a
Serre functor.

Finally, no other power of $\Sigma$ is a Serre functor of $\Dc(B)$:
If $\Sigma^i$ is a Serre functor then $\Sigma^i \simeq \Sigma^w$
whence $\Sigma^{i-w} \simeq \id$.  This implies $i = w$ since already
$\Sigma^{i-w} B \cong B$ implies $i = w$ as one sees by taking
homology. 
\end{proof}

\begin{Remark}
\label{rmk:tau}
The AR translation of $\sT$ is $\tau = \Sigma^{-1}S = \Sigma^{w - 1} =
\Sigma^d$.
\end{Remark}

\begin{Proposition}
\label{pro:ARquiver}
\begin{enumerate}

  \item  If $w \neq 1$ then the AR quiver of $\sT$ consists of $|d|$
  copies of $\BZ A_{\infty}$.  One copy is shown in Figure \ref{fig:3}
  and the others are obtained by applying $\Sigma$, $\Sigma^2$,
  $\ldots$, $\Sigma^{|d|-1}$.

\smallskip

  \item  If $w = 1$ then the AR quiver of $\sT$ consists of countably
  many homogeneous tubes.  One tube is shown in Figure \ref{fig:4} and
  the others are obtained by applying all non-zero powers of $\Sigma$.

\end{enumerate}

\begin{figure}
\[
  \xymatrix @-3.5pc @! {
    & \vdots \ar[dr] & & \vdots \ar[dr] & & \vdots \ar[dr] & & \vdots \ar[dr] & & \vdots \ar[dr] & & \vdots & \\
    \cdots \ar[dr]& & \Sigma^0 X_3 \ar[ur] \ar[dr] & & \Sigma^{-d} X_3 \ar[ur] \ar[dr] & & \Sigma^{-2d} X_3 \ar[ur] \ar[dr] & & \Sigma^{-3d}X_3 \ar[ur] \ar[dr] & & \Sigma^{-4d}X_3 \ar[ur] \ar[dr] & & \cdots \\
    & \Sigma^d X_2 \ar[ur] \ar[dr] & & \Sigma^0 X_2 \ar[ur] \ar[dr] & & \Sigma^{-d} X_2 \ar[ur] \ar[dr] & & \Sigma^{-2d}X_2 \ar[ur] \ar[dr] & & \Sigma^{-3d}X_2 \ar[ur] \ar[dr] & & \Sigma^{-4d}X_2 \ar[ur] \ar[dr] & \\
    \cdots \ar[ur]\ar[dr]& & \Sigma^d X_1 \ar[ur] \ar[dr] & & \Sigma^0 X_1 \ar[ur] \ar[dr] & & \Sigma^{-d} X_1 \ar[ur] \ar[dr] & & \Sigma^{-2d}X_1 \ar[ur] \ar[dr] & & \Sigma^{-3d}X_1 \ar[ur] \ar[dr] & & \cdots\\
    & \Sigma^{2d} X_0 \ar[ur] & & \Sigma^d X_0 \ar[ur] & & \Sigma^0 X_0 \ar[ur] & & \Sigma^{-d}X_0 \ar[ur] & & \Sigma^{-2d}X_0 \ar[ur] & & \Sigma^{-3d}X_0 \ar[ur] & \\
               }
\]
\caption{A component of the AR quiver for $w \neq 1$}
\label{fig:3}
\end{figure}
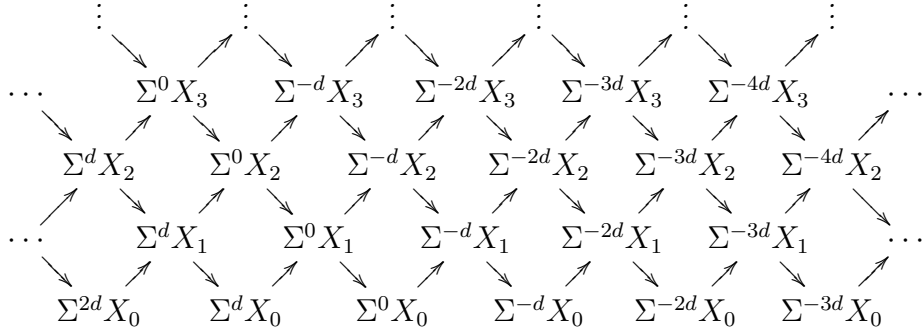

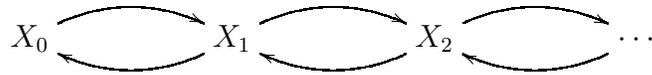
\begin{figure}
\[
\vcenter{
  \xymatrix{
    X_0 \ar@/^1pc/[rr] & & 
    X_1 \ar@/^1pc/[rr] \ar@/^1pc/[ll] & &
    X_2 \ar@/^1pc/[rr] \ar@/^1pc/[ll] & &
    \cdots \ar@/^1pc/[ll]
           }
        }
\]
\caption{A component of the AR quiver for $w = 1$}
\label{fig:4}
\end{figure}

\end{Proposition}

\begin{proof}
For $w \geq 2$ this is \cite[thm.\ 8.13]{J}.

For $w$ general, the shape of the AR quiver is given in \cite[sec.\
3.3]{FY}.  For $w \leq 0$, to see that the $|d|$ copies of $\BZ
A_{\infty}$ look as claimed, one can compute the AR triangles of $\sT$
by methods similar to those of \cite[sec.\ 8]{J}.

Finally, for $w = 1$ we have $d = 0$.  The AR translation is $\tau =
\Sigma^0 = \operatorname{id}$ by Remark \ref{rmk:tau}, so for each
$X_r$ there is an AR triangle $X_r \rightarrow Y \rightarrow X_r$.
The long exact homology sequence shows that if $r = 0$ then $Y = X_1$
and if $r \geq 1$ then $Y = X_{r - 1} \oplus X_{r + 1}$.  Hence the
homogeneous tube in Figure \ref{fig:4} is a component of the AR quiver
as claimed.  For each $i$, applying $\Sigma^i$ to Figure \ref{fig:4}
gives a component of the AR quiver.  The components obtained in this
fashion contain all indecomposable objects of $\sT$ so form the whole
AR quiver.
\end{proof}

\section{Morphisms in $\sT$}
\label{sec:morphisms}

This section computes the $\Hom$ spaces between indecomposable
objects in the category $\sT$.

\begin{Definition}
Suppose that $w \neq 1$ so the AR quiver of $\sT$ consists of copies
of $\BZ A_{\infty}$ by Proposition \ref{pro:ARquiver}(i).  Let $t \in
\sT$ be an indecomposable object.  Figure \ref{fig:5} defines two sets
$F^{\pm}(t)$ consisting of indecomposable objects in the same
component of the AR quiver as $t$.  Each set can be described as a
rectangle stretching off to infinity in one direction; it consists of
the objects inside the indicated boundaries including the ones on the
boundaries.  In particular we have $t \in F^{\pm}(t)$.
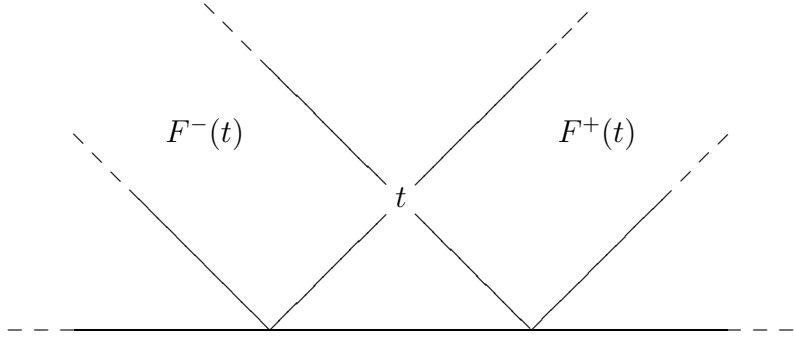
\begin{figure}
\[
\vcenter{
  \xymatrix @-3.0pc @! {
    &&&*{} &&&&&&&& \\
    &&&& *{} \ar@{--}[ul] & & & & *{} \ar@{--}[ur] \\
    &*{}&& F^-(t) & & & & & & F^+(t) && *{}\\
    &&*{}\ar@{--}[ul]& & & & {t} \ar@{-}[ddll] \ar@{-}[uull] \ar@{-}[ddrr] \ar@{-}[uurr]& & &&*{}\ar@{--}[ur]&\\ 
    && \\
    *{}\ar@{--}[r]&*{} \ar@{-}[rrr] && & *{}\ar@{-}[uull]\ar@{-}[rrrrrr]& & & & \ar@{-}[uurr]\ar@{-}[rrr]&&&*{}\ar@{--}[r]&*{}\\
           }
}
\]
\caption{The regions $F^{\pm}(t)$ for $w \neq 1$}
\label{fig:5}
\end{figure}
\end{Definition}

\begin{Proposition}
\label{pro:Hom_neq01}
Suppose that $w \neq 0, 1$.  Let $t, u$ be indecomposable objects in
$\sT$.  Then
\[
  \dim_k \sT( t , u ) = 
  \left\{
    \begin{array}{cl}
      1 & \mbox{for $u \in F^+(t) \cup F^-(St)$}, \\[2pt]
      0 & \mbox{otherwise},
    \end{array}
  \right.
\]
where $S = \Sigma^w$ is the Serre functor of $\sT$; see Proposition
\ref{pro:Serre}.
\end{Proposition}

\begin{Proposition}
\label{pro:Hom_0}
Suppose that $w = 0$.  Let $t, u$ be indecomposable objects in
$\sT$.  Then
\[
  \dim_k \sT( t , u ) = 
  \left\{
    \begin{array}{cl}
      2 & \mbox{for $u = t$}, \\[2pt]
      1 & \mbox{for $u \in ( F^+(t) \cup F^-(t) ) \setminus t$}, \\[2pt]
      0 & \mbox{otherwise}.
    \end{array}
  \right.
\]
\end{Proposition}

\begin{Proposition}
\label{pro:Hom_1}
Suppose that $w = 1$.  Let $u$ be an indecomposable object of $\sT$.
Then
\[
  \dim_k \sT(X_r,u) = 
  \left\{
    \begin{array}{cl}
      \min \{\, r , s \,\} + 1 & \mbox{for $u = X_s$ or $u = \Sigma X_s$}, \\[2pt]
      0 & \mbox{for all other $u$}.
    \end{array}
  \right.
\]
\end{Proposition}

Note that in Proposition \ref{pro:Hom_neq01}, the sets $F^+(t)$ and
$F^-(St)$ are disjoint.  For $w \neq 2$, they even sit in different
components of the AR quiver, while for $w = 2$ we have $d = w - 1 = 1$
and the AR quiver has only one component.  In Proposition
\ref{pro:Hom_0}, the sets $F^+(t)$ and $F^-(t)$ have intersection $t$.
In this case $w = 0$ so $d = w - 1 = -1$ and the AR quiver has only
one component.

{\it Proof } of Propositions \ref{pro:Hom_neq01}, \ref{pro:Hom_0}, and
\ref{pro:Hom_1}.

Propositions \ref{pro:Hom_neq01} and \ref{pro:Hom_0} give the
dimensions of $\Hom$ spaces in a conceptual way using the regions
$F^{\pm}$.  Unfortunately we do not have a conceptual proof.

The proof we have is pedestrian: Applying $\RHom_A( - , X_s )$ to the
distinguished triangle from Remark \ref{rmk:Xr} gives a new
distinguished triangle whose long exact homology sequence contains
\[
  \H^{i-1}(X_s)
  \rightarrow \H^{i - (r+1)d - 1}(X_s)
  \rightarrow \H^i \RHom_A( X_r , X_s )
  \rightarrow \H^{i}(X_s)
  \rightarrow \H^{i - (r+1)d}(X_s).
\]
The middle term is isomorphic to $\sT( X_r , \Sigma^i X_s)$.  The DG
module $X_s$ is $A / (T^{s + 1})$, so the four outer terms are easily
computable.  The first and last maps are induced by $\cdot T^{r + 1}$
and can also be computed.  Hence the middle term can be determined.

For $w \neq 1$, combining the dimensions of $\Hom$ spaces with the
detailed structure of the AR quiver as described by Proposition
\ref{pro:ARquiver} proves Propositions \ref{pro:Hom_neq01} and
\ref{pro:Hom_0}, and for $w = 1$ one gets Proposition
\ref{pro:Hom_1} directly.
\hfill $\Box$

\section{t- and co-t-structures}
\label{sec:remarks}

This section gives some easy properties of t- and co-t-structures.
Lemmas \ref{lem:2} and \ref{lem:3} are valid in general triangulated
categories.

Recall that if $( \sX , \sY )$ is a t-structure then the heart is $\sH
= \sX \cap \Sigma \sY$, and if $( \sA , \sB )$ is a co-t-structure
then the co-heart is $\sC = \sA \cap \Sigma^{-1} \sB$.

\begin{Lemma}
\label{lem:2}
Let $( \sX , \sY )$ be a t-structure and $( \sA , \sB )$ a
co-t-structure with heart and co-heart $\sH$ and $\sC$.
\begin{enumerate}

  \item  $\Hom( \sH , \Sigma^{<0} \sH ) = 0$.

\smallskip

  \item  $\Hom( \sC , \Sigma^{>0} \sC ) = 0$. 

\smallskip

  \item  $\sX = \Sigma \sX \;\Leftrightarrow\; \sH = 0$.

\smallskip

  \item  $\sA = \Sigma \sA \;\Leftrightarrow\; \sC = 0$.

\end{enumerate}
\end{Lemma}

\begin{proof}
(i)  We have $\sH \subseteq \sX$ and $\Sigma^{<0}\sH \subseteq
\Sigma^{<0}\Sigma\sY = \Sigma^{\leq 0}\sY \subseteq \sY$.
The last $\subseteq$ is a well known property of t-structures and
follows from $\Sigma^{\geq 0}\sX \subseteq \sX$ by taking right
perpendicular categories; cf.\ \cite[remark after def.\ 2.2]{IY} by
which $\sX^{\perp} = \sY$.  Here $\perp$ is as defined in \cite[start
of sec.\ 2]{IY}.  But $\Hom( \sX , \sY ) = 0$ so $\Hom( \sH ,
\Sigma^{<0}\sH ) = 0$ follows.

(ii)  Dual to part (i).

(iii)  $\Rightarrow$:  Suppose $h \in \sH$.  Then $\Sigma^{-1}h \in
\Sigma^{-1}\sH \subseteq \Sigma^{-1}\sX = \sX$ so $h =
\Sigma(\Sigma^{-1} h) \in \Sigma \sX$.  We also have $h \in \sH
\subseteq \Sigma \sY$.  But $\Hom(\Sigma \sX , \Sigma \sY) = \Hom( \sX
, \sY ) = 0$ so $\Hom( h , h ) = 0$ proving $h = 0$.

$\Leftarrow$:  For $x \in \sX$ consider the distinguished triangle
\begin{equation}
\label{equ:tt_triangle}
  x' \rightarrow \Sigma^{-1}x \rightarrow y'
\end{equation}
with $x' \in \sX$, $y' \in \sY$ which exists because $( \sX , \sY )$
is a torsion pair.  It gives a distinguished triangle $x \rightarrow
\Sigma y' \rightarrow \Sigma^2 x'$ with $x, \Sigma^2 x' \in \sX$.  But
$\sX$ is closed under extensions since it is equal to ${}^{\perp}\sY$
by \cite[remark after def.\ 2.2]{IY} again, so $\Sigma y' \in \sX$.

We also have $\Sigma y' \in \Sigma \sY$ so $\Sigma y' \in \sH$ and
hence $\Sigma y' = 0$.  But then the distinguished triangle
\eqref{equ:tt_triangle} shows $\Sigma^{-1}x \cong x' \in \sX$.  Hence
$\Sigma^{-1}\sX \subseteq \sX$, and since we also know $\Sigma\sX
\subseteq \sX$ it follows that $\Sigma\sX = \sX$.

(iv) Dual to part (iii).
\end{proof}

A torsion pair $( \sM , \sN )$ with $\Sigma \sM = \sM$ (and
consequently $\Sigma \sN = \sN$) is called a stable t-structure; see
\cite[p.\ 468]{M}.  In this case, $\sM$ and $\sN$ are thick
subcategories of $\sT$.

\begin{Lemma}
\label{lem:3}
If $( \sM , \sN )$ and $( \sM' , \sN' )$ are torsion pairs with $\sM
\subseteq \sM'$ and $\sN \subseteq \sN'$, then  $( \sM , \sN ) = (
\sM' , \sN' )$. 
\end{Lemma}

\begin{proof}
The inclusion $\sN \subseteq \sN'$ implies ${}^{\perp}\sN \supseteq
{}^{\perp}\sN'$, but this reads $\sM \supseteq \sM'$ by \cite[remark
after def.\ 2.2]{IY} so we learn $\sM = \sM'$.  Hence also $\sN =
\sM^{\perp} = \sM'^{\perp} = \sN'$. 
\end{proof}

\begin{Lemma}
\label{lem:stable}
A stable t-structure in $\sT$ is trivial.
\end{Lemma}

\begin{proof}
Let $( \sX , \sY )$ be a stable t-structure in $\sT$ with $\sX \neq
0$.  Then $\sX$ contains an indecomposable object $x$.  But $\sX$ is a
thick subcategory of $\sT$, and it is easy to see from the AR quiver
of $\sT$ that hence $\sX = \sT$.
\end{proof}

\section{Proof of Theorem A}
\label{sec:proof}

\subsection{Proof of Theorem A for t-structures, $w \leq -1$}

$\,$
\medskip
\newline
Let $( \sX , \sY )$ be a t-structure in $\sT$ with heart $\sH = \sX
\cap \Sigma \sY$ and let $h \in \sH$.  Serre duality gives
\[
  \Hom_k ( \sT( h , h ) , k )
  \cong \sT( h , Sh )
  \cong \sT( h , \Sigma^w h )
  = 0
\]
where ``$=0$'' is by Lemma \ref{lem:2}(i) because $w \leq -1$.  This
implies $h = 0$ so $\sH = 0$.  But then $( \sX , \sY )$ is a stable
t-structure by Lemma \ref{lem:2}(iii) and hence trivial by Lemma
\ref{lem:stable}.

\subsection{Proof of Theorem A for t-structures, $w = 0$}

$\,$
\medskip
\newline
In this case we have $d = w - 1 = -1$.  The AR quiver consists of $|d|
= 1$ copy of $\BZ A_{\infty}$ by Lemma \ref{pro:ARquiver}(i); see
Figure \ref{fig:3}.

Assume that $( \sX , \sY )$ is a non-trivial t-structure in $\sT$.  By
Lemmas \ref{lem:stable} and \ref{lem:2}(iii) the heart $\sH$ is
non-zero so contains an indecomposable object.

However, if $t$ is an indecomposable object not on the base line of
the AR quiver then $\tau t \in F^-(t)$; see Figure \ref{fig:5}.  Hence
$\sT( t , \tau t ) \neq 0$ by Proposition \ref{pro:Hom_0}, and by
Remark \ref{rmk:tau} this reads $\sT ( t , \Sigma^{-1}t ) \neq 0$.
But $\sT( \sH , \Sigma^{<0}\sH ) = 0$ by Lemma \ref{lem:2}(i), so each
indecomposable object in $\sH$ is forced to be on the base line of the
AR quiver.

Moreover, if $h \in \sH$ is indecomposable then $\sH$ cannot contain
another indecomposable object $h'$:  Both objects would have to be on
the base line of the AR quiver which has only one component, so we
would have $h' = \tau^i h$ for some $i \neq 0$, that is, $h' =
\Sigma^{-i}h$.  But this contradicts $\sT( \sH , \Sigma^{<0}\sH ) =
0$.

It follows that $\sH = \add(h)$ for an indecomposable object $h$ on
the base line of the AR quiver, and $h = \Sigma^i X_0$ for some $i$.
Direct computation shows that $h$ is $0$-spherical, so there is a
non-zero, non-invertible morphism $h \rightarrow h$.  But this
morphism is easily verified not to have a kernel in $\sH$, and this is
a contradiction since the heart of a t-structure is abelian.

\subsection{Proof of Theorem A for t-structures, $w = 1$}

$\,$
\medskip
\newline
Here the AR quiver of $\sT$ consists of countably many stable tubes as
detailed in Proposition \ref{pro:ARquiver}(ii).

By \cite[sec.\ 3.1]{FY}, an alternative model of $\sT$ is
$\Df(k\mbox{\textlbrackdbl} X \mbox{\textrbrackdbl})$, the derived
category of complexes with bounded finite length homology over the
ring $k\mbox{\textlbrackdbl} X \mbox{\textrbrackdbl}$.  This shows
that $\sT$ has a canonical t-structure.

Assume that $( \sX , \sY )$ is a non-trivial t-structure in $\sT$.  In
particular, $\sX$ is closed under extensions.  The components of the
AR quiver of $\sT$ are homogeneous tubes and the AR triangles of $\sT$
can be read off.  The triangles imply that if $\sX$ contains an
indecomposable object $t$ then it contains the whole component of $t$.  So
$\sX$ is equal to $\add$ of a collection of components of the AR
quiver.  Now let $Q$ be a component such that $Q \subseteq \sX$ but
$\Sigma^{-1}Q, \Sigma^{-2}Q, \ldots \not\subseteq \sX$.  Such a $Q$
exists because $\sX$ is closed under $\Sigma$ and not equal to $0$ or
$\sT$.  It is then clear that
\[
  \sX = \add( Q \cup \Sigma Q \cup \cdots ).
\]
The right hand side only depends on the component $Q$ of the AR
quiver, and since all other components have the form $\Sigma^i Q$ (see
Proposition \ref{pro:ARquiver}(ii)), this implies that all non-trivial
t-structures are (de)suspensions of each other, and hence
(de)suspensions of the canonical t-structure.

\subsection{Proof of Theorem A for t-structures, $w \geq 2$}

$\,$
\medskip
\newline
Recall that $A$ is $k[T]$ viewed as a DG algebra with $T$ in
homological degree $d = w - 1$ and zero differential.  Each object of
$\sT$ is a direct sum of finitely many (de)suspensions of the objects
$X_r = A / (T^{r+1})$ by Remark \ref{rmk:Krull-Schmidt} and
Proposition \ref{pro:Xr}.  In particular, each object of $\sT$ is
isomorphic to a DG module $t$ which is finite dimensional over $k$.

Since $w \geq 2$ we have $d \geq 1$ which means that $A$ is a chain DG
algebra.  So for each DG left-$A$-module $t$ there is a distinguished
triangle $t_{(\geq 0)} \rightarrow t \rightarrow t_{(<0)}$ in $\sD(A)$
induced by the following (vertical) short exact sequence of DG
modules.
\[
  \xymatrix{
    \cdots \ar[r] & t_2 \ar[r] \ar@{=}[d] & t_1 \ar[r] \ar@{=}[d] & \Ker \partial_0 \ar[r] \ar@{^{(}->}[d] & 0 \ar[r] \ar[d] & 0 \ar[r] \ar[d] & \cdots\\
    \cdots \ar[r] & t_2 \ar[r] \ar[d] & t_1 \ar[r] \ar[d] & t_0 \ar[r]^{\partial_0} \ar@{->>}[d] & t_{-1} \ar[r] \ar@{=}[d] & t_{-2} \ar[r] \ar@{=}[d] & \cdots\\
    \cdots \ar[r] & 0 \ar[r] & 0 \ar[r] & t_0 / \Ker \partial_0 \ar[r] & t_{-1} \ar[r] & t_{-2} \ar[r] & \cdots
           }
\]
Each of $t_{(\geq 0)}$ and $t_{(<0)}$ is also finite dimensional over
$k$ which implies that they can be built in finitely many steps from
the DG module $k$; that is, they belong to $\langle k \rangle = \sT$.
Hence the distinguished triangle is in $\sT$, so $( \sT_{ ( \geq 0 ) }
, \sT_{ ( <0 ) } )$ is a t-structure in $\sT$ where
\begin{align}
\nonumber
  \sT_{ ( \geq 0 ) } & = \{\, t \in \sT 
  \,|\, \mbox{ $H_*(t)$ is in homological degrees $\geq 0$ } \,\}, \\
\label{equ:canonical_co-t-structure}
  \sT_{ ( < 0 ) } & = \{\, t \in \sT
  \,|\, \mbox{ $H_*(t)$ is in homological degrees $< 0$ } \,\}.
\end{align}
We refer to this t-structure as canonical.  Note that it was first
constructed, in higher generality, in \cite[thm.\ 1.3]{HKM}, and in
the DG case in \cite[lem.\ 2.2]{A} and \cite[lem.\ 5.2]{KY}.

It is easy to check that $\sT_{ ( \geq 0 ) }$ is the smallest
subcategory of $\sT$ which contains $X_0$ and is closed under
$\Sigma$, extensions, and direct summands.  Likewise, $\sT_{ ( < 0 )
}$ is the smallest subcategory of $\sT$ which contains $\Sigma^{-1}
X_0$ and is closed under $\Sigma^{-1}$, extensions, and direct
summands.

Now assume that $( \sX , \sY )$ is a non-trivial t-structure in $\sT$.
By Lemmas \ref{lem:stable} and \ref{lem:2}(iii) the heart $\sH$ is
non-zero so contains an indecomposable object $h$.  We have $\sT( h ,
\Sigma^{<0}h ) = 0$ by Lemma \ref{lem:2}(i).

If $t$ is an indecomposable object not on the base line of the AR
quiver then $\tau^{-1} t \in F^+(t)$; see Figure \ref{fig:5}.  Hence
$\sT( t , \tau^{-1} t ) \neq 0$ by Proposition \ref{pro:Hom_neq01},
and by Remark \ref{rmk:tau} this reads $\sT ( t , \Sigma^{-d}t ) \neq
0$.  Hence $h$ is forced to be on the base line of the AR quiver.
Suspending or desuspending the t-structure, we can assume $h = X_0$.

We have $h \in \sX$ and $h \in \Sigma\sY$ whence $\Sigma^{-1}h \in
\sY$.  That is, $X_0 \in \sX$ and $\Sigma^{-1} X_0 \in \sY$.

However, $\sX$ is closed under $\Sigma$, extensions, and direct
summands, and since $\sT_{ ( \geq 0 ) }$ is the smallest subcategory
of $\sT$ with these properties which contains $X_0$, we get $\sT_{ (
  \geq 0 ) } \subseteq \sX$.  Similarly, $\sT_{ ( < 0 )} \subseteq
\sY$.

By Lemma \ref{lem:3} this forces $( \sX , \sY ) = ( \sT_{ ( \geq 0 ) }
, \sT_{ ( < 0 ) } )$, and we have shown that as desired, up to
(de)suspension, any non-trivial t-structure in $\sT$ is the canonical
one.

\subsection{Proof of Theorem A for co-t-structures, $w \leq 0$}

$\,$ \medskip \newline In the proof for t-structures, $w \geq 2$, we
showed a canonical t-structure.  Tweaking the method slightly in the
present case produces a canonical co-t-structure.  Each object of
$\sT$ is still isomorphic to a DG module $t$ which is finite
dimensional over $k$.  Since $A$ is $k[T]$ with $T$ in homological
degree $d = w - 1$, and since $w \leq 0$ and $d \leq -1$, we have that
$A$ is a cochain DG algebra.  So there is a distinguished triangle
$t_{\leq 0} \rightarrow t \rightarrow t_{>0}$ in $\sD(A)$ where the
subscripts indicate hard truncations in the relevant homological
degrees.  Each of $t_{\leq 0}$ and $t_{>0}$ is also finite dimensional
over $k$ and is therefore in $\sT$.  Hence $( \sT_{ ( \leq 0 ) } ,
\sT_{ ( >0 ) } )$ is a co-t-structure in $\sT$ where
\begin{align*}
  \sT_{ ( \leq 0 ) } & = \{\, t \in \sT 
  \,|\, \mbox{ $H_*(t)$ is in homological degrees $\leq 0$ } \,\}, \\
  \sT_{ ( > 0 ) } & = \{\, t \in \sT
  \,|\, \mbox{ $H_*(t)$ is in homological degrees $> 0$ } \,\}.
\end{align*}
The rest of the proof is dual to the proof for t-structures, $w \geq
2$.

\subsection{Proof of Theorem A for co-t-structures, $w \geq 1$}

$\,$
\medskip
\newline
This is dual to the proof for t-structures, $w \leq -1$.

\medskip
\noindent
{\bf Acknowledgement.}
Holm and J\o rgensen were supported by the research priority programme
SPP 1388 {\em Darstellungstheorie} of the Deutsche
Forschungsgemeinschaft (DFG).  They gratefully acknowledge financial
support through the grant HO 1880/4-1.

We thank Changjian Fu for the observation reproduced at the end of the
introduction.


\begin{thebibliography}{19}

\bibitem{AI}  T.\ Aihara and O.\ Iyama, Silting mutation in
  triangulated categories, preprint (2010).  {\tt
    math.RT/1009.3370v2. }

\bibitem{A}  C.\ Amiot, {\it Cluster categories for algebras of global
dimension $2$ and quivers with potential}, Ann.\ Inst.\ Fourier
(Grenoble) {\bf 59} (2009), 2525--2590.

\bibitem{BBD} A.\ A.\ Beilinson, J.\ Bernstein, and P.\ Deligne, {\it
Faisceaux pervers}, Ast\'{e}risque {\bf 100} (1982) (Vol.\ 
1 of the proceedings of the conference ``Analysis and topology on
singular spaces'', Luminy, 1981).

\bibitem{B}  M.\ V.\ Bondarko, {\it Weight structures vs.\
t-structures; weight filtrations, spectral sequences, and complexes
(for motives and in general)}, J.\ K-Theory {\bf 6} (2010), 387--504. 

\bibitem{FY}  C.\ Fu and D.\ Yang, The Ringel-Hall Lie algebra of a
  spherical object, preprint (2011).  {\tt math.RT/1103.1241v2.}

\bibitem{H}  D.\ Happel, {\it On the derived category of a
finite-dimensional algebra}, Comment.\ Math.\ Helv.\ {\bf 62} (1987),
339--389.

\bibitem{HJ}  T.\ Holm and P.\ J\o rgensen, {\it On a cluster category
of infinite Dynkin type, and the relation to triangulations of the
infinity-gon}, Math.\ Z., in press.  {\tt DOI:
10.1007/s00209-010-0797-z.} 

\bibitem{HKM}  M.\ Hoshino, Y.\ Kato, and J.\ Miyachi, {\it On
    t-structures and torsion theories induced by compact objects}, J.\
  Pure Appl.\ Algebra {\bf 167} (2002), 15--35.

\bibitem{IY}  O.\ Iyama and Y.\ Yoshino, {\it Mutation in
triangulated categories and rigid Cohen-Macaulay mo\-du\-les},
Invent.\ Math.\ {\bf 172} (2008), 117--168.

\bibitem{J}  P.\ J\o rgensen, {\it Auslander-Reiten theory over
    topological spaces}, Comment.\ Math.\ Helv.\ {\bf 79} (2004),
  160--182. 

\bibitem{K}  B.\ Keller, {\it Deformed Calabi--Yau completions}, with
an appendix by Michel Van den Bergh, J.\ Reine Angew.\ Math.\ {\bf
654} (2011), 125--180. 

\bibitem{KY}  B.\ Keller and D.\ Yang, {\it Derived equivalences from
mutations of quivers with potential}, Adv.\ Math.\ {\bf 226} (2011),
2118--2168. 

\bibitem{KYZ}  B.\ Keller, D.\ Yang, and G.\ Zhou, {\it The Hall
    algebra of a spherical object}, J.\ London Math.\ Soc.\ (2) {\bf
    80} (2009), 771--784.

\bibitem{M}  J.\ Miyachi, {\it Localization of triangulated categories
    and derived categories}, J.\ Algebra {\bf 141} (1991), 463--483.

\bibitem{Ng}  P.\ Ng, A characterization of torsion theories in the
  cluster category of type $A_{\infty}$, preprint (2010).  {\tt
    math.RT/1005.4364v1.} 

\bibitem{P}  D.\ Pauksztello, {\it Compact corigid objects in
triangulated categories and co-t-structures}, Cent.\ Eur.\ J.\ Math.\
{\bf 6} (2008), 25--42.

\bibitem{R}  C.\ M.\ Ringel, ``Tame algebras and quadratic
forms'',  Lecture Notes in Math., Vol.\ 1099, Springer, Berlin, 1984. 

\bibitem{ZZ}  Y.\ Zhou and B.\ Zhu, Mutation of torsion pairs in
triangulated categories and its geometric realization, preprint
(2011).  {\tt math.RT/1105.3521v1.}

\end{thebibliography}
\end{document}